\documentclass[11pt]{amsart}

\usepackage{amsfonts, amsmath, amsthm, amssymb, epsfig, bm, graphics, psfrag, latexsym, chngcntr, color}
\usepackage[all,cmtip]{xy}
\usepackage{tikz} \usetikzlibrary{matrix,arrows}
\definecolor{darkblue}{rgb}{0,0,0.4}
\usepackage[colorlinks=true, citecolor=darkblue, filecolor=darkblue,
linkcolor=darkblue, urlcolor=darkblue]{hyperref}
\usepackage[all]{hypcap}

\newtheorem{thm}{Theorem}[section]         
\newtheorem{lem}[thm]{Lemma}

\counterwithin{figure}{section}

\newcommand{\R}{\mathbb{R}}
\newcommand{\Z}{\mathbb{Z}}
\newcommand{\F}{\mathbb{F}}

\newcommand{\OO}{\text{O}}
\newcommand{\XX}{\text{X}}

\newcommand{\mc}{\mathcal}

\newcommand{\mf}{\mathfrak}

\newcommand{\wt}{\widetilde}

\newcommand{\del}{\partial}
\newcommand{\sbs}{\subset}
\newcommand{\sbseq}{\subseteq}
\newcommand{\sm}{\setminus}

\newcommand{\al}{\alpha}
\newcommand{\be}{\beta}

\newcommand{\Zoltan}{Zolt\'{a}n}
\newcommand{\Szabo}{Szab\'{o}}
\newcommand{\Ozsvath}{Ozsv\'{a}th}

\newcommand{\ith}{^{\text{th}}}

\newcommand{\Id}{\operatorname{Id}}

\newcommand{\CF}{\mathit{CF}}

\begin{document}

\title{Grid diagrams and the Ozsv\'{a}th-Szab\'{o} tau-invariant}
\author{Sucharit Sarkar}
\address{Department of Mathematics\\Columbia University\\New York, NY 10027}
\email{\href{mailto:sucharit.sarkar@gmail.com}{sucharit.sarkar@gmail.com}}

\subjclass[2010]{\href{http://www.ams.org/mathscinet/search/mscdoc.html?code=57M25}{57M25}}

\date{}

\begin{abstract}
  We use grid diagrams to investigate the \Ozsvath-\Szabo{} concordance invariant $\tau$, and to prove that $\left|\tau(K_1)-\tau(K_2)\right|\leq g$, whenever there is a genus $g$ knot 
cobordism joining $K_1$ to $K_2$. This leads to an entirely grid diagram-based proof of Kronheimer-Mrowka's theorem, formerly known as the Milnor conjecture.
\end{abstract}

\maketitle

\section{Introduction}

Links inside $S^3$ can be represented by a combinatorial structure called grid diagrams, and these grid diagrams can then be used to study various properties of the links. Grid diagrams first appeared as arc-presentations in \cite{gridHB}, and are also equivalent to the square-bridge positions of \cite{gridHCL}, the Legenedrian realisations of \cite{gridHM}, the asterisk presentations of \cite{gridLN} and the fences of \cite{gridLR}. They have been used to classify essential tori in the complement of non-split links \cite{gridJBWM}, to define certain Legendrian and transverse knot invariants \cite{POZSzDT}, and to describe an algorithm to detect the unknot \cite{gridID}. Many properties of grid diagrams have been studied in great detail in \cite{gridPC}.

Quite recently, it was observed in \cite{CMPOSS} that grid diagrams can also be used to study a family of knot invariants and link invariants called knot Floer homology, originally defined for knots in \cite{POZSzknotinvariants, JR}, and extended for links in \cite{POZSzlinkinvariants}. Knot Floer homology is a powerful knot invariant, which generalises the Alexander polynomial and can detect the knot genus \cite{POZSzgenusbounds}, and fiberedness \cite{YN}.  However, it was originally defined using holomorphic geometry, and it is an interesting endeavor to find combinatorial reinterpretations of various aspects of the theory using grid diagrams.

The aspect of knot Floer homology that we will study here is the \Ozsvath-\Szabo{} knot invariant $\tau$, as defined in \cite{JR, POZSz4ballgenus}. It was shown in \cite{POZSz4ballgenus} that the absolute value of $\tau$ is a lower bound for the four-ball genus, and it can be used to prove a theorem due to Kronheimer and Mrowka \cite{generalPKTM}, formerly known as the Milnor conjecture, that the unknotting number of the torus knot $T(p,q)$ is $\frac{(p-1)(q-1)}{2}$. In this paper, we will study the definition of $\tau$ in terms of grid diagrams, we will compute $\tau$ for torus knots using grid diagrams, and we will give a grid diagram-based proof of the fact that $\left|\tau\right|$ is a lower bound for the four-ball genus which similar in spirit to Rasmussen's proof for the $s$ invariant \cite{generalJR}, thereby giving a new combinatorial proof of the Milnor's conjecture.

\section{Knot cobordisms}

Throughout this paper, the terms \emph{knots} and \emph{links} will mean oriented knots and oriented links inside $S^3$. A \emph{link   cobordism} from a link $L_1$ to a link $L_2$ is a properly embedded oriented surface $S$ inside $S^3\times [0,1]$, such that $\del S\cap(S^3\times\{0\})=-L_1\times\{0\}$ and $\del S\cap(S^3\times\{1\})=L_2\times\{1\}$. A link cobordism joining a knot to another knot is called a \emph{knot   cobordism}. After a small isotopy of the cobordism $S$ inside $S^3\times [0,1]$ relative to the boundary, we can assume that the second projection $p_2\colon S^3\times [0,1]\rightarrow\R$, restricted to $S$, is a Morse function. We will call this Morse function $p_2|_{S}$, the \emph{time function} $t$. The index $0$, index $1$ and index $2$ critical points of the time function are called \emph{births}\phantomsection\label{birth}, \emph{saddles}\phantomsection\label{saddle} and \emph{deaths}\phantomsection\label{death}.  In a saddle, either two link components \emph{merge}\phantomsection\label{merge} to form a single link component, or a link component \emph{splits}\phantomsection\label{split} to form two link components.  A link cobordism $S$ joining $L_1$ to $L_2$ is called a \emph{concordance} if $S$ is homeomorphic to $L_1\times[0,1]$. A concordance where the time function does not have any critical points is called an \emph{isotopy}. Two links are said to be \emph{isoptic} if there is an isotopy joining them. Two links are said to be \emph{concordant} if there is a concordance joining them.

There is a concordance invariant $\tau$ for knots \cite{POZSz4ballgenus}, such that if there is a connected genus $g$ cobordism from $K_1$ to $K_2$, then $\left|\tau(K_1)-\tau(K_2)\right|\leq g$. The \emph{four-ball genus} of a knot $K$ is the smallest integer $g^*(K)$ such that there is a connected genus $g^*(K)$ cobordism from $K$ to the unknot. The \emph{unknotting number} $u(K)$ of a knot $K$ is the smallest number of crossing changes that needs to be done to convert it to the unknot. However, a crossing change is a particular type of a connected genus $1$ cobordism. Therefore, for any knot $K$, we have $u(K)\geq g^*(K)\geq\left|\tau(K)\right|$.

In the subsequent sections, we will start with the definition of
$\tau$ using grid diagrams, and directly prove that the inequality
$\left|\tau(K_1)-\tau(K_2)\right|\leq g$ holds whenever there is a
connected genus $g$ cobordism joining $K_1$ to $K_2$. Representing
torus knots by grid diagrams, we will compute $\tau$ and produce an
explicit unknotting sequence, and thereby give a new and completely
grid diagram-based proof of Milnor's conjecture
$u(T(p,q))=\frac{(p-1)(q-1)}{2}$. As a preparatory move, let us prove
the following lemma about knot cobordisms.

\begin{lem}\label{lem:isotopy}
  If $S$ is a connected knot cobordism from $K_1$ to $K_2$, then $S$
  can be isotoped to a cobordism $S'$ inside $S^3\times[0,1]$ relative
  to the boundary, preserving the number of births, saddles and deaths
  throughout the isotopy, such that for $S'$, all the births happen at
  time $\frac{1}{4}$, all the saddles happen at time $\frac{1}{2}$,
  and all the deaths happen at time $\frac{3}{4}$. Furthermore, we can
  ensure that $S'$ restricted to $S^3\times[0,\frac{1}{4})$ and $S'$
  restricted to $S^3\times(\frac{3}{4},1]$ are both product
  cobordisms.
\end{lem}

\begin{proof}
  We would like to think of cobordisms as movies with time running
  from $0$ to $1$. The still at time $t$ is $S\cap(S^3\times\{t\})$;
  therefore, all but finitely many of the stills are links in
  $S^3$. We start with some link, and as the movie plays, for most of
  the time, we simply isotope the link. However, at finitely many
  points in time, we can have births, saddles or deaths, as shown in
  \hyperref[fig:morse]{Figure \ref*{fig:morse}}.

\begin{figure}
\begin{center}
\begin{tabular}{ccc}
\epsfig{file=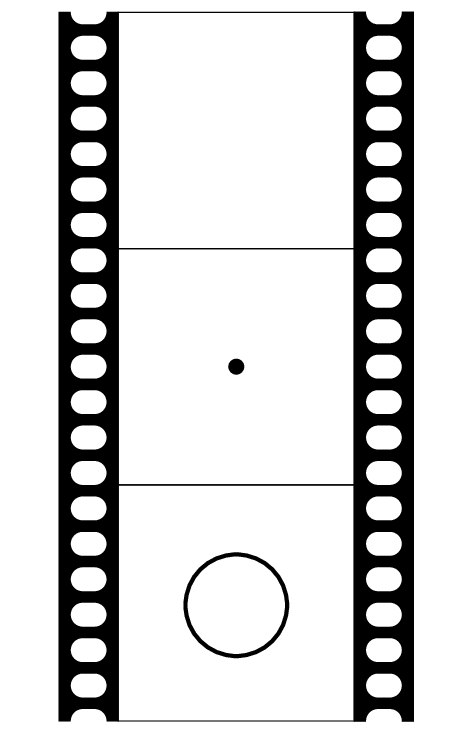,height=0.2\textheight,clip=}&
\epsfig{file=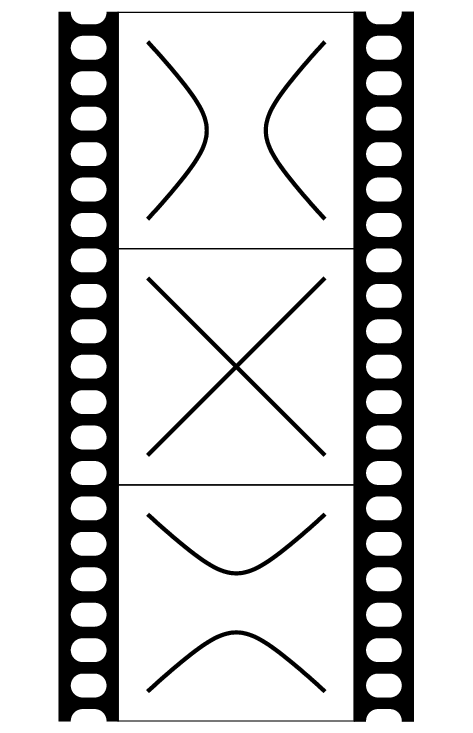,height=0.2\textheight,clip=}&
\epsfig{file=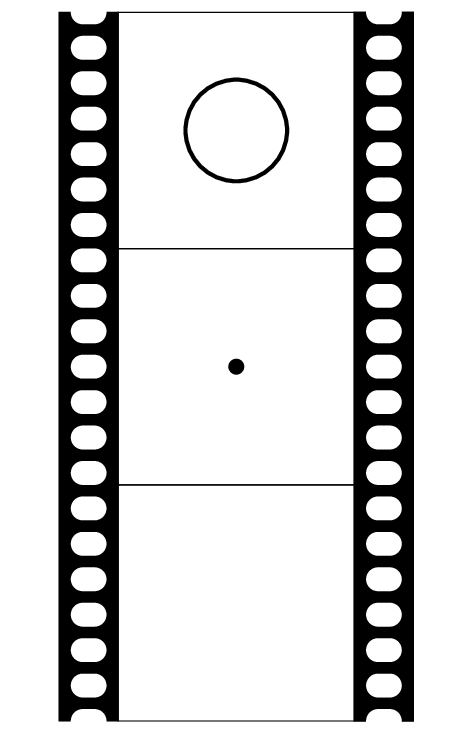,height=0.2\textheight,clip=}
\end{tabular}
\end{center}
\caption{Birth, saddle and death.}\label{fig:morse}
\end{figure}

Clearly, we can isotope $S$, while preserving the number of births,
saddles and deaths, to ensure that all the deaths happen at the very
end. Just before some death is about to happen, intervene, and keep
the relevant unknot component alive. Since the unknot, thus kept
alive, is an unknot supported inside a very small ball, it behaves
like a point, and therefore generically does not interfere with the
rest of the cobordism. Therefore, we can postpone all the deaths,
until all that is left of the cobordism is a product cobordism, and
then have all the deaths. Similarly, we can ensure that all the births
happen at the very beginning. Thus we can assume that the cobordism
$S$ restricted to either $S^3\times [0,\frac{1}{4})$ or
  $S^3\times(\frac{3}{4},1]$ is a product cobordism, and all the
births happen at time $\frac{1}{4}$, and all the deaths happen at time
$\frac{3}{4}$.

Now we want to isotope the cobordism inside
$S^3\times[\frac{1}{4},\frac{3}{4}]$, relative to the boundary, so as
to ensure that all the saddles happen at the same time. By
reparametrizing time if necessary, assume that all the saddles happen
before time $t=\frac{1}{2}$. We will now describe how to postpone all
the saddles until that point, and then make all the saddles happen
simultaneously.

\begin{figure}
\begin{center}
\includegraphics[height=0.2\textheight]{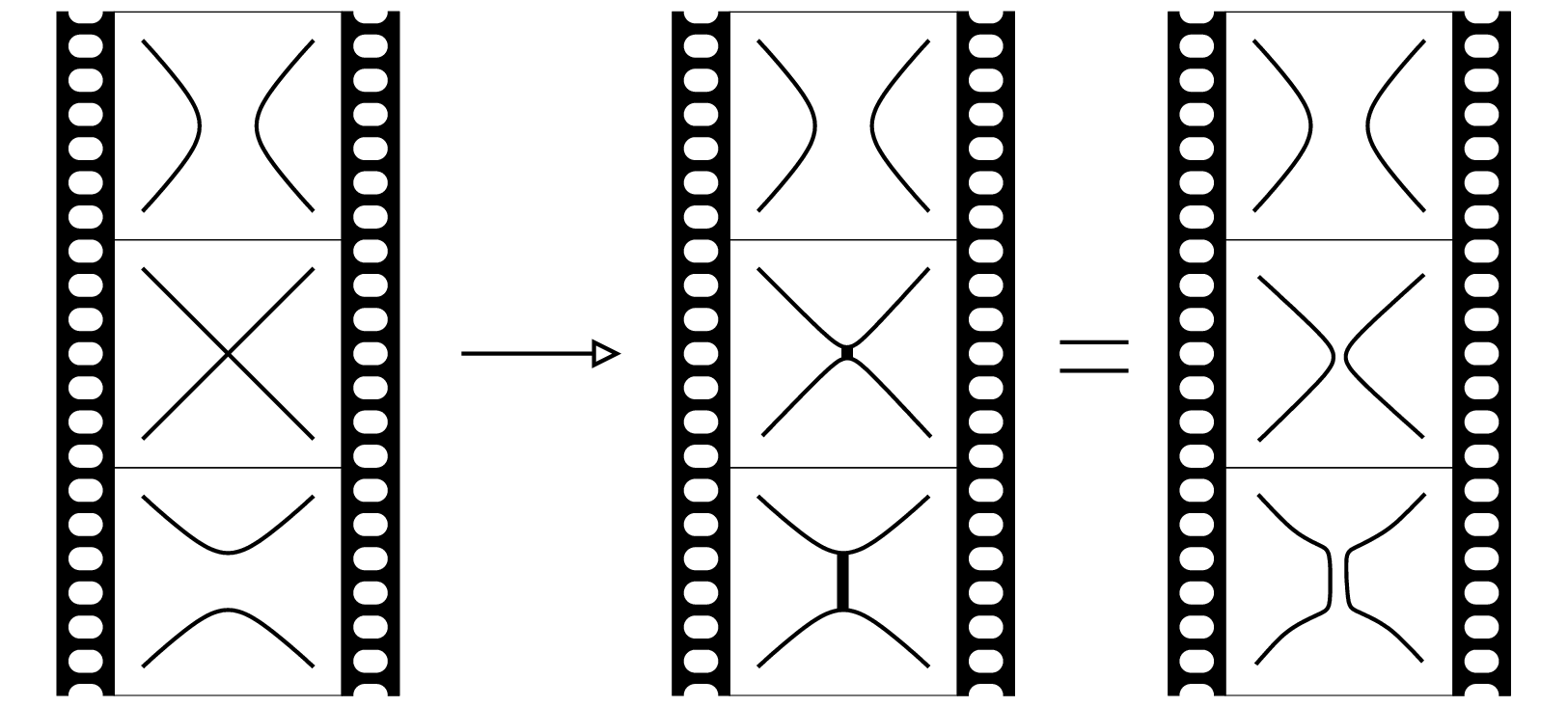}
\end{center}
\caption{Modifying the cobordism $S$ to get the cobordism $S'$ by
  delaying a saddle.}\label{fig:redband}
\end{figure}

During the movie for $S$, if at some point in time a saddle happens,
then in the movie for the new cobordism $S'$, immediately after that
point, attach an untwisted band to the two strands near the saddle and
add an $1$-handle to the link along that band. Therefore, in the
modified cobordism $S'$, the saddle has not yet taken place. However,
while watching the movie for $S'$, if we ignore all these new bands
and the associated $1$-handles, it will look exactly like the movie
for $S$.  This modification from $S$ to $S'$ is shown in
\hyperref[fig:redband]{Figure \ref*{fig:redband}}, with the bands
being denoted by thick lines.

In the movie for $S'$, move the endpoints of the bands as prescribed
by the movie for $S$, and move the bands themselves in any fashion,
while ensuring that they stay disjoint from each other and from the
rest of the link. Then at time $t=\frac{1}{2}$, after we have
encountered all the saddles of $S$, and after we have attached bands
for each one of them in $S'$, actually do all the saddles for
$S'$. The saddles have the effect of cutting open all the bands, which
can then deformation retract to their endpoints on the link. After
that, the movie for $S'$ agrees the movie for $S$.
\end{proof}

\section{Grid diagrams}

The best reference for this section is \cite{CMPOZSzDT}. Many of the definitions and theorems that we are about to mention in this section, are treated in great detail in that paper. However, for the sake of completeness, let us still go through some of the basic definitons and state some of the basic properties of grid diagrams.

\subsection{Grid diagrams for \texorpdfstring{$S^3$}{S3}}
An \emph{index-$n$ $S^3$-grid diagram} $\mf{G}=(T,\al,\be,O)$ is a picture of the following type on a torus $T$: $\al$ and $\be$ are two $n$-component embedded multicurves on $T$; each $\al$-circle is transverse to each $\be$-circle, and they intersect each other at one point; the $n$ components of $T\sm\al$ are called the \emph{horizontal annuli}; the $n$ components of $T\sm\be$ are called the \emph{vertical annuli}; $O$ is a formal sum of $n$ markings on $T$, such that each horizontal annulus contains one $O$-marking and each vertical annulus contains one $O$-marking; for some $0\leq k\leq n$, exactly $(n-k)$ of the $O$-markings are designated \emph{special} and are often denoted by $\varnothing$; the other $O$-markings are called \emph{normal} $O$-markings are numbered $O_1,\ldots,O_k$.

A \emph{generator} $x$ is a formal sum of $n$ points on $T$, often called \emph{$x$-coordinates}, such that each $\al$-circle contains one $x$-coordinate and each $\be$-circle contains one $x$-coordinate. The set of all the $n!$ generators is denoted by $\mc{G}_{\mf{G}}$. Given two generators $x,y\in\mc{G}_{\mf{G}}$ which differ in exactly two coordinates, a \emph{rectangle} joining $x$ to $y$ is an embedded rectangle $R\sbs T$ such that: the sides of $R$ lie on $\al\cup\be$; the top-right and bottom-left corners of $R$ are $x$-coordinates and the top-left and bottom-right corners of $R$ are $y$-coordinates, or in other words, $\del(\del R|_{\al})=y-x$; $R$ does not contain any other $x$-coordinates; and $R$ does not contain any special $O$-marking. For $x,y\in\mc{G}_{\mf{G}}$, the set $\mc{R}_{\mf{G}}(x,y)$ is defined to be empty if $x$ and $y$ do not differ in exactly two coordinates, or else, it is defined to be the set of all rectangles joining $x$ to $y$. Given a rectangle $R\in\mc{R}_{\mf{G}}(x,y)$, the number $n_i(R)$ is defined to be $1$ if $R$ contains $O_i$, and is defined to be $0$ otherwise.

To each generator we can associate an integer-valued grading $M$
called the \emph{Maslov grading} in the following way: the torus is
cut up along some $\al$-circle and some $\be$-circle and identified
with $[0,n)\times[0,n)$, such that the $\al$-circles become the lines
$[0,n)\times\{i\}$ for $i\in\{0,\ldots,n-1\}$, and the $\be$-circles
become the lines $\{i\}\times[0,n)$ for $i\in\{0,\ldots,n-1\}$; let
$\mc{J}$ be the bilinear form on the singular $0$-chains of $\R^2$,
such that, if $p=(p_1,p_2)$ and $q=(q_1,q_2)$ are two points in
$\R^2$, $\mc{J}(p,q)=\frac{1}{2}$ if $(p_1-q_1)(p_2-q_2)>0$ and
$\mc{J}(p,q)=0$ otherwise; for any generator $x\in\mc{G}_{\mf{G}}$,
the Maslov grading is defined as $M(x)=\mc{J}(x-O,x-O)+1$.

The \emph{grid chain complex} $\CF_{\mf{G}}$ over $\F_2$ is defined in the following way: it is freely generated over $\F_2[U_1,\ldots, U_k]$ by the elements of $\mc{G}_{\mf{G}}$; the Maslov grading is extended to $\CF_{\mf{G}}$ by declaring the Maslov grading of each $U_i$ to be $(-2)$; the homological grading is simply the Maslov grading; the boundary map $\del$ is $U_i$-equivariant and for any $x\in\mc{G}_{\mf{G}}$, it is given by 
$$\del x=\sum_{y\in\mc{G}_{\mf{G}}}y\sum_{R\in\mc{R}_{\mf{G}}(x,y)}\prod_i U_i^{n_i(R)}.$$

\begin{thm}\cite{CMPOSS, CMPOZSzDT}\label{thm:basicproperty}
  If $k$, the number of normal $O$-markings, is less than $n$, then
  the homology of $\CF_{\mf{G}}$ is isomorphic, as graded
  $\F_2[U_1,\ldots,U_k]$-modules, to $\otimes^{n-k-1}(\F_2\oplus\F_2[-1])$, where
  each $U_i$ acts trivially on the right hand side, each $\F_2$ lives
  in grading zero and $[i]$ denotes a grading shift by $i$.
\end{thm}

In light of the above theorem, the number $n-k-1$ is often called the \emph{smallest Maslov grading}\phantomsection\label{smallestgrading} since it is the smallest grading in which the homology of $\CF_{\mf{G}}$ is supported; furtheremore, the rank of the homology in the smallest Maslov grading is always one.

\subsection{Grid diagrams for links}
An \emph{index-$n$ link-grid diagram} $\mf{L}=(T,\al,\be,O,X)$ is a picture on a torus $T$ such that: $X$ is a formal sum of $n$ points on the torus; if $\mf{f}(\mf{L})=(T,\al,\be,O)$ is the diagram obtained from $\mf{L}$ by forgetting the $X$-markings, then $\mf{f}(\mf{L})$ is an index-$n$ $S^3$-grid diagram; furthermore, each horizontal annulus contains some $X$-marking and each vertical annulus contains some $X$-marking.  

Given an index-$n$ link-grid diagram $\mf{L}$, we can produce $n^2$ links in the following way: cut the torus $T$ along some $\al$-circle and some $\be$-circle to identify it with $[0,n)\times[0,n)$; in every horizontal strip, join the $X$-marking to the $O$-marking by a horizontal line segment, and in every vertical strip join the $O$-marking to the $X$-marking by a vertical line segment, with the understanding that if there is a square containing both an $X$-marking and an $O$-marking, then we put a small unknot in that square; and finally at every crossing, declare the vertical segment to be the overpass. These $n^2$ links, thus obtained, are all isotopic to one another; therefore, a link-grid diagram represents a link isotopy class. Whenever we say that a link $L$ is represented by a link-grid diagram $\mf{L}$, we mean that $\mf{L}$ represents the link isotopy class that contains $L$. Call $\mf{L}$ \emph{tight}\phantomsection\label{tight}, if every link component in $\mf{L}$ contains exactly one special $O$-marking.

\begin{lem}\cite{gridPC}\label{lem:cromwelleasy}
Every link can be represented by link-grid diagrams.
\end{lem}

In a link-grid diagram, generators can be endowed with a
$\frac{1}{2}\Z$-valued grading $A$ called the \emph{Alexander grading}
as follows: the torus is cut up along an $\al$-circle and $\be$-circle
such that it can be identified with $[0,n)\times[0,n)$; if $\mc{J}$ is
the bilinear form on the $0$-chains of $\R^2$ from before, then for
any generator $x\in\mc{G}_{\mf{f}(\mf{L})}$, the Alexander grading is
defined as 
\begin{align*}
A(x)&=\mc{J}(x-\frac{1}{2}(X+O),X-O)-\frac{n-1}{2}\\
&=\mc{J}(x,X)-\mc{J}(x,O)-\frac{1}{2}\mc{J}(X,X)+\frac{1}{2}\mc{J}(O,O)-\frac{n-1}{2}.
\end{align*}
This
can be extended to an Alexander grading on $\CF_{\mf{f}(\mf{L})}$ by
declaring the Alexander grading of each $U_i$ to be $(-1)$. An astute
reader will observe that our definition of Alexander grading differs
from the usual definition of Alexander grading \cite{CMPOZSzDT} by an
additive constant of $\frac{l-1}{2}$, where $l$ is the number of link
components; therefore, the two definitions agree for knots. The
boundary map $\del$ does not increase this Alexander grading. This
leads to the following definition of the \emph{Alexander filtration}
on the grid chain complex: for every $a\in\frac{1}{2}\Z$, the
filtration level $\mc{F}_{\mf{L}}(a)\sbseq \CF_{\mf{f}(\mf{L})}$ is
defined to be the subcomplex supported in Alexander grading $a$ or
less.

Define $\tau_{\mf{L}}$ to be smallest $a\in\frac{1}{2}\Z$ such that the map induced on homology from the inclusion $\mc{F}_{\mf{L}}(a)\hookrightarrow \CF_{\mf{f}(\mf{L})}$ is non-trivial.

\begin{thm}\cite{CMPOSS}
If a knot $K$ is represented by a tight link-grid diagram $\mf{L}$, then $\tau(K)=\tau_{\mf{L}}$.
\end{thm}

This is actually very close to the original definition of $\tau$. Combining this fact with the main result from \cite{POZSz4ballgenus}, we get the following:

\begin{thm}\cite{POZSz4ballgenus, CMPOSS}\label{thm:main}
For tight link-grid diagrams $\mf{L}$ that represent knots,
$\tau_{\mf{L}}$ depends only on the isotopy class of the knot. If we
define $\tau(K)$ to be equal to $\tau_{\mf{L}}$ for any tight
link-grid diagram representing $K$, then
$\left|\tau(K_1)-\tau(K_2)\right|\leq g$, whenever there is a
connected genus $g$ knot cobordism from $K_1$ to $K_2$.
\end{thm}

The proof of this theorem requires the holomorphic techniques of \cite{POZSz4ballgenus}. We will bypass those methods, and give a new proof of this theorem using only grid diagrams. That is one of our main results.

\subsection{Moves on \texorpdfstring{$S^3$}{S3}-grid diagrams}
In this subsection, we will describe certain \emph{$S^3$-grid moves} which convert an $S^3$-grid diagram $\mf{G}_1$ to another $S^3$-grid diagrams $\mf{G}_2$, and in each case, we will define chain maps from $\CF_{\mf{G_1}}$ to $\CF_{\mf{G}_2}$. Given a link cobordism from a link $L$ represented by a link-grid diagram $\mf{L}$ to a link $L'$ represented by a link-grid diagram $\mf{L}'$, we will be able to construct a sequence of link-grid diagrams $\mf{L}=\mf{L}_0,\mf{L}_1,\ldots,\mf{L}_{m-1},\mf{L}_m=\mf{L}'$, such that for each $i$, $\mf{f}(\mf{L}_i)$ and $\mf{f}(\mf{L}_{i+1})$ will be related by one of the following $S^3$-grid moves. Therefore, we will have chain maps $\CF_{\mf{f}(\mf{L}_i)}\rightarrow \CF_{\mf{f}(\mf{L}_{i+1})}$, and by composing, we will get a chain map from $\CF_{\mf{f}(\mf{L})}$ to $\CF_{\mf{f}(\mf{L}')}$, which we will use to find a relation between $\tau_{\mf{L}}$ and $\tau_{\mf{L}'}$.

\subsubsection*{$S^3$-grid move (1)}\label{subsub:s3identity} $\mf{G}_1=\mf{G}_2$. In this case, we define the chain map to be identity, which clearly preserves Maslov grading and is a quasi-isomorphism.

\subsubsection*{$S^3$-grid move (2)}\label{subsub:s3renumbering} $\mf{G}_2$ is obtained from $\mf{G}_1$ by renumbering the normal $O$-markings. If there are exactly $k$ normal $O$-markings, which are renumbered by some permutation $\sigma\in\mf{S}_k$, then the chain map sends $\prod_i U_i^{m_i} x$ to $\prod_i U_{\sigma(i)}^{m_i} x$. Although this map is not $U_i$-equivariant, it preserves the Maslov grading and is a quasi-isomorphism. 

\subsubsection*{$S^3$-grid move (3)}\label{subsub:s3commutation} $\mf{G}_2$ is obtained from $\mf{G}_1$ by a \emph{commutation}. There are two types of commutations: a horizontal commutation where we  interchange the $O$-markings in two adjacent horizontal annuli, or a vertical commutation where we interchange the $O$-markings in two adjacent vertical annuli. In either case, we represent both $\mf{G}_1$ and $\mf{G}_2$ by a single diagram $\mf{G}$ on the torus $T$, and the chain map is defined by counting certain pentagons in $\mf{G}$. For example, in a horizontal commutation, as illustrated in \hyperref[fig:commutation]{Figure \ref*{fig:commutation}}, the shaded pentagon contributes to the chain map. As shown in \cite[Section 3.1]{CMPOZSzDT}, the pentagon map also preserves the Maslov grading and is a quasi-isomorphism.

\begin{figure}
\psfrag{x}{$x$}
\psfrag{y}{$y$}
\psfrag{o}{$\OO$}
\psfrag{a}{$\al_2$}
\psfrag{a'}{$\al_1$}
\begin{center}
\includegraphics[width=0.5\textwidth]{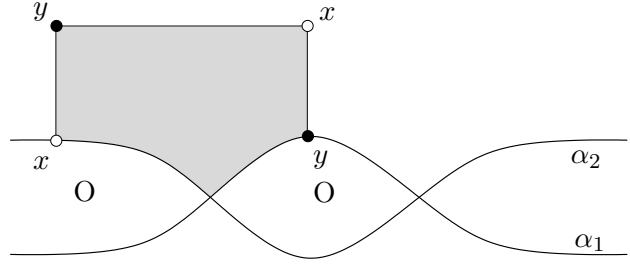}
\end{center}
\caption{The diagram $\mf{G}$. The grid diagram $\mf{G}_i$ is obtained from $\mf{G}$ by deleting the the circle labeled $\al_i$. If $x\in\mc{G}_{\mf{G}_1}$ is represented by the white circles and if $y\in\mc{G}_{\mf{G}_2}$ is represented by the black circles, then the shaded pentagon contributes a coefficient of $y$ for the chain map evaluated at $x$.}\label{fig:commutation}
\end{figure}

\subsubsection*{$S^3$-grid move (4)}\label{subsub:s3normaldestab} $\mf{G}_2$ is obtained from $\mf{G}_1$ by a \emph{normal   stabilisation} or a \emph{normal destabilisation}. In a normal destabilisation, we start with an index-$(n+1)$ $S^3$-grid diagram $\mf{G}_1$ with exactly $(k+1)$ normal $O$-markings, and we get the index-$n$ $S^3$-grid diagram $\mf{G}_2$ by deleting $O_{k+1}$ and then deformation retracting the closure of the horizontal annulus through $O_{k+1}$ to an $\al$-circle and deformation retracting the closure of the vertical annulus through $O_{k+1}$ to a $\be$-circle. A normal stabilisation is the reverse process of a normal destabilisation. Let us assume that $\mf{G}_2$ is obtained from $\mf{G}_1$ by a normal destabilisation; we will describe four chain maps: two chain maps $d_{11}$ and $d_{22}$ from $\CF_{\mf{G}_1}$ to $\CF_{\mf{G}_2}$, and two chain maps $s_{11}$ and $s_{22}$ from $\CF_{\mf{G}_2}$ to $\CF_{\mf{G}_1}$.

\begin{figure}
\psfrag{o}{$\varnothing$}
\psfrag{ok}[][][0.8]{$\OO_{k+1}$}
\psfrag{a0}{$\al_0$}
\psfrag{a1}{$\al_1$}
\psfrag{a2}{$\al_2$}
\psfrag{b0}{$\be_0$}
\psfrag{b1}{$\be_1$}
\psfrag{b2}{$\be_2$}
\begin{center}
\begin{tabular}{cc}
\epsfig{file=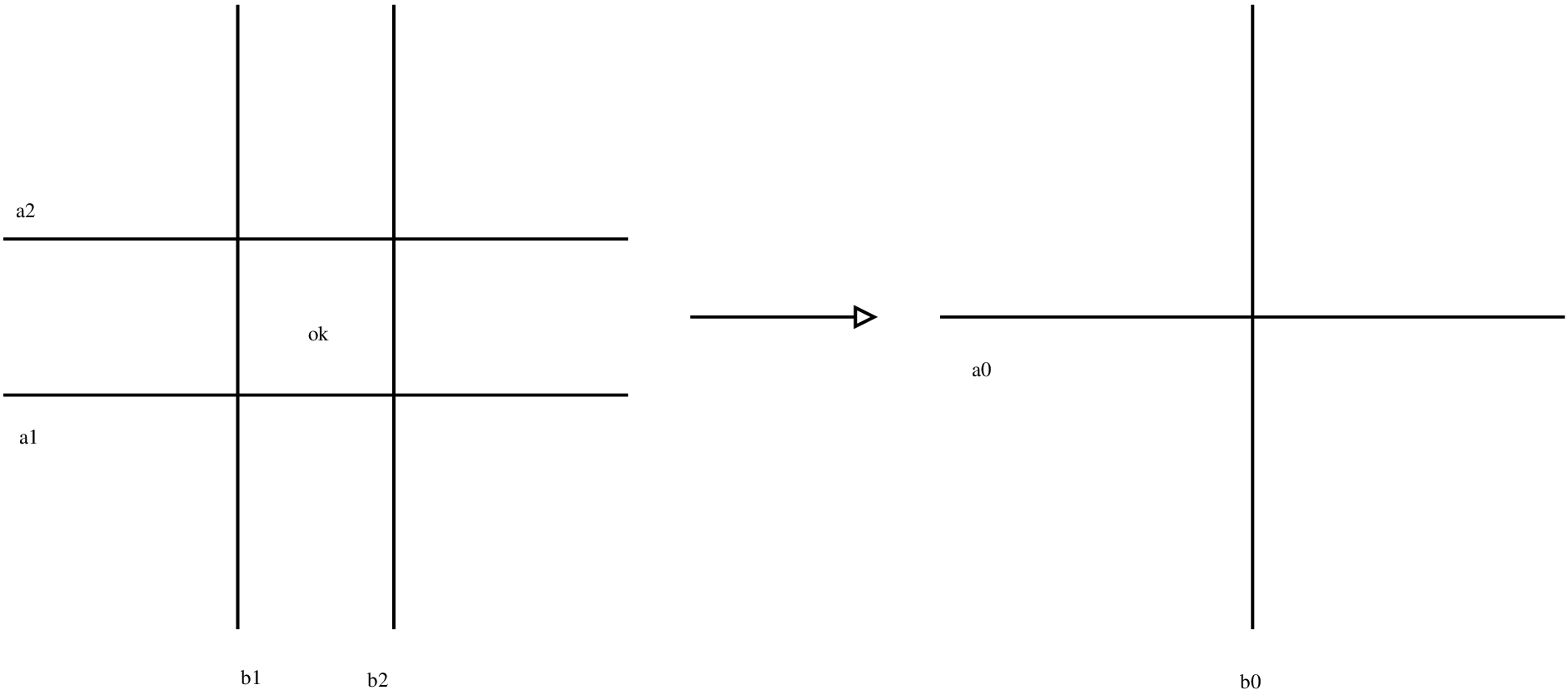,width=0.45\textwidth,clip=}&
\epsfig{file=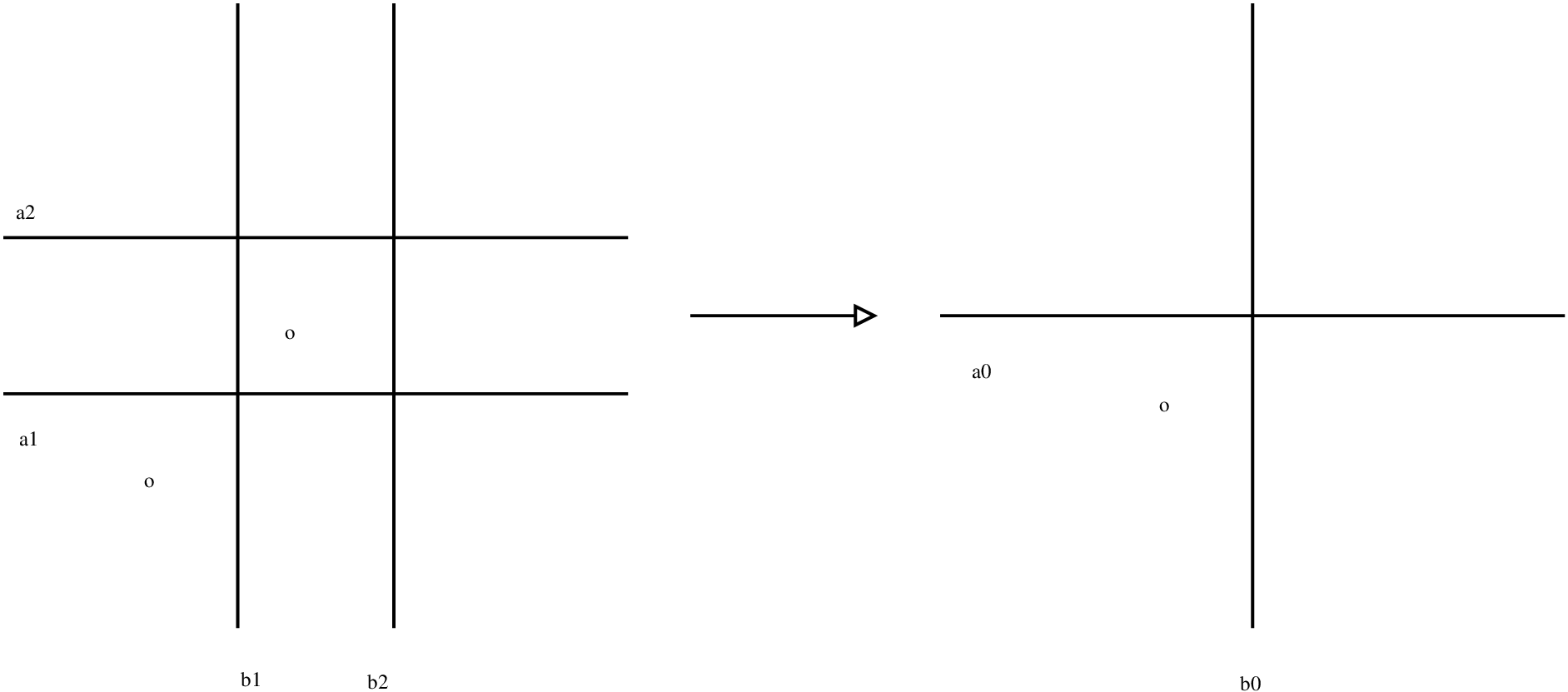,width=0.45\textwidth,clip=}
\end{tabular}
\end{center}
\caption{The two types of destabilisations.}\label{fig:destab}
\end{figure}

In the $S^3$-grid diagram $\mf{G}_2$, let the $\al$-circle which comes from deformation retracting the horizontal annulus be numbered $\al_0$, and let the $\be$-circle which comes from deformation retracting the vertical annulus be numbered $\be_0$. In the $S^3$-grid diagram $\mf{G}_1$, let the two $\al$-circles on the boundary of that horizontal annulus be numbered $\al_1$ and $\al_2$, such that $\al_1$ lies just below $\al_2$, and let the two $\be$-circles on the boundary of that vertical annulus be numbered $\be_1$ and $\be_2$, such that $\be_1$ lies just to left of $\be_2$. This is shown in the first part of \hyperref[fig:destab]{Figure \ref*{fig:destab}}.

Let us define four injective maps
$F_{ij}\colon\mc{G}_{\mf{G}_2}\rightarrow\mc{G}_{\mf{G}_1}$, for
$i,j\in\{1,2\}$. For defining $F_{ij}$, we identify the $\al$-circles
of $\mf{G}_1$, except $\al_i$, with the $\al$-circles of $\mf{G}_2$ in
the natural way, and we identify the $\be$-circles of $\mf{G}_2$,
except $\be_j$, with the $\be$-circles of $\mf{G}_2$ in the natural
way. Under these identifications, a generator $x\in\mc{G}_{\mf{G}_2}$
produces a formal sum of $n$ points in $\mf{G}_1$. We define
$F_{ij}(x)$ to be that formal sum plus $\al_i\cap\be_j$.

The destabilisation map $d_{11}$ is precisely the snail map $F^R$ as defined in \cite[Section 3.2]{CMPOZSzDT}. It is a homomorphism of $\F_2[U_1,\ldots,U_k]$-modules, and for $x\in\mc{G}_{\mf{G}_1}$, it is defined as
$$d_{11}(U_{k+1}^m x)=U^m\sum_{y\in\mc{G}_{\mf{G}_2}}y\sum_{D\in\mc{S}_1(x,F_{11}(y),\al_1\cap\be_1)}\prod_{1\leq   i\leq k} U_i^{n_i(D)}$$
where: $U=0$ if the horizontal annulus just below $\al_0$ contains a special $O$-marking, and $U=U_j$ if the horizontal annulus contains the normal $O$-marking $O_j$; $\mc{S}_1(x,z,p)$ is the set of all Type $(1)$ snail-like domains centered at $p$ joining $x$ to $z$, as illustrated in the bottom row of \cite[Figure 13]{CMPOZSzDT} or the top row of \hyperref[fig:snail]{Figure \ref*{fig:snail}}; and $n_i(D)$ is the number of times $D$ passes through $O_i$. As shown in \cite{CMPOZSzDT}, this map preserves the Maslov grading and is a quasi-isomorphism.

\begin{figure}
\psfrag{p}{$p$}
\begin{center}
\includegraphics[width=\textwidth]{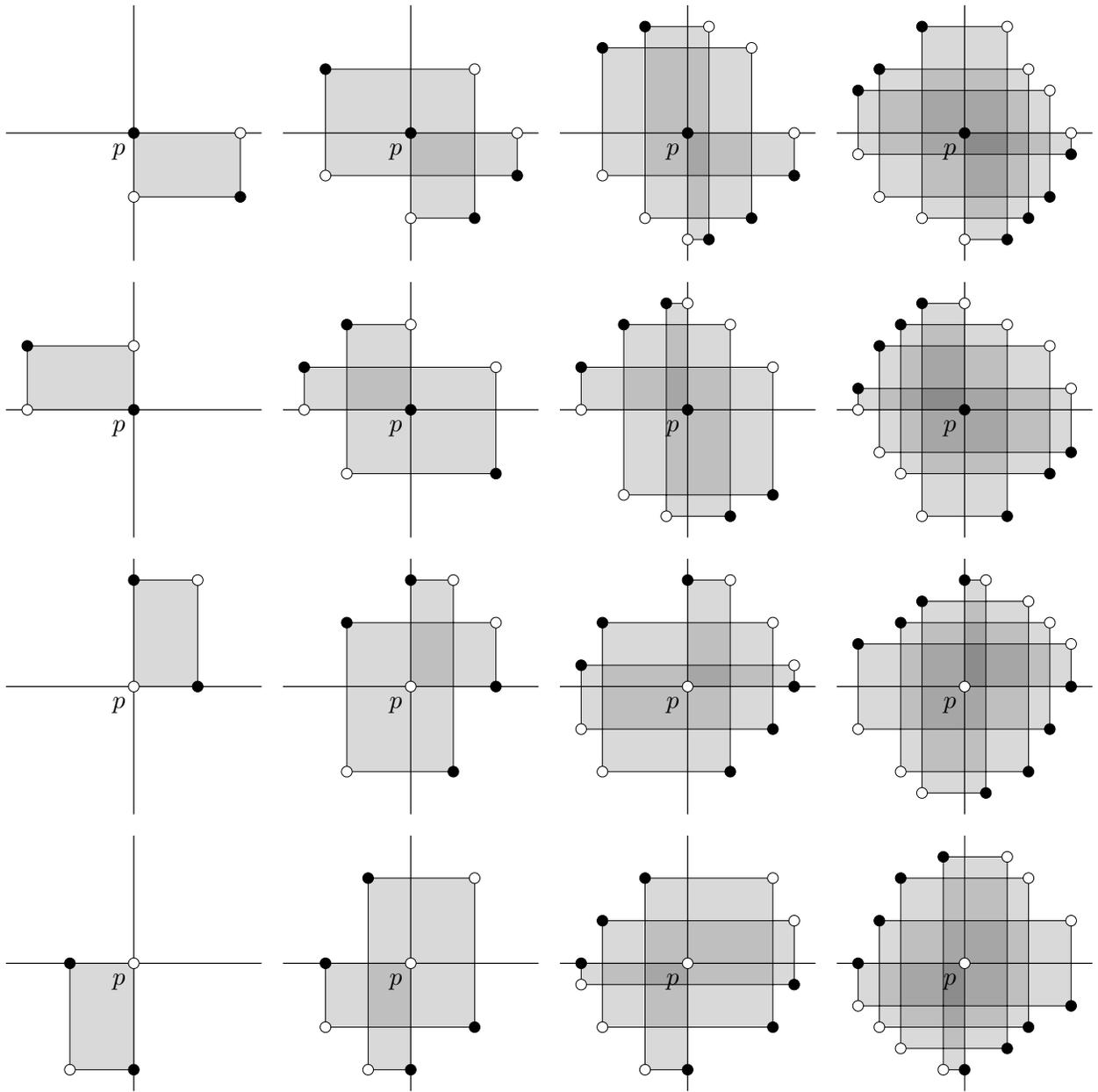}
\end{center}
\caption{The different types of snail-like domains. The $i\ith$ row displays some of the domains in $\mc{S}_i(x,z,p)$, where the $x$- and $z$-coordinates that are not disjoint from the domains are represented by the white and black circles, respectively. We always assume that none of the snail-like domains pass through any of the special $O$-markings.}\label{fig:snail}
\end{figure}

The destabilisation map $d_{22}$ is obtained by rotating all the diagrams by $180^{\circ}$. Stated more precisely,
$$d_{22}(U_{k+1}^m x)=U^m\sum_{y\in\mc{G}_{\mf{G}_2}}y\sum_{D\in\mc{S}_2(x,F_{22}(y),\al_2\cap\be_2)}\prod_{1\leq   i\leq k} U_i^{n_i(D)}$$
where: $U=0$ if the horizontal annulus just above $\al_0$ contains a special $O$-marking, and $U=U_j$ if the horizontal annulus contains the normal $O$-marking $O_j$; $\mc{S}_2(x,z,p)$ is the set of all Type $(2)$ snail-like domains centered at $p$ joining $x$ to $z$, as illustrated in the second row of \hyperref[fig:snail]{Figure \ref*{fig:snail}}; and $n_i(D)$ is the number of times $D$ passes through $O_i$.

The two stabilisation maps are defined similarly. Namely, for $x\in\mc{G}_{\mf{G}_2}$ and $j\in\{1,2\}$,
$$s_{jj}(x)=\sum_{y\in\mc{G}_{\mf{G}_1}}y\sum_{D\in\mc{S}_{j+2}(F_{jj}(x),y,\al_j\cap\be_j)}\prod_{1\leq   i\leq k} U_i^{n_i(D)}.$$
The proofs of Lemma 3.5 and Proposition 3.8 of \cite{CMPOZSzDT} go through after rotating all the diagrams by $\pm 90^{\circ}$. Therefore, the stabilisation maps are also chain maps; furthermore, they preserve the Maslov grading and are quasi-isomorphisms.

\subsubsection*{$S^3$-grid move (5)}\label{subsub:s3specialdestab} $\mf{G}_2$ is obtained from $\mf{G}_1$ by a \emph{special   destabilisation}.  This is like a normal destabilisation, except we assume that the $O$-marking that is being deleted is a special $O$-marking. We also assume that the square immediately to the bottom-left of this special $O$-marking also contains a special $O$-marking. The situation is illustrated in the second part of \hyperref[fig:destab]{Figure \ref*{fig:destab}}. We will reuse the notations that we had used to describe a normal destabilisation.

A \emph{hexagon} $H$ is an embedded hexagon in $\mf{G}_1$ with boundary lying in $\al\cup\be$, which has one $270^{\circ}$ angle at $\al_1\cap\be_1$ and five other $90^{\circ}$ angles, such that $H$ contains the special $O$-marking that is being deleted, but not the special $O$-marking that lies to the bottom-left of it; $H$ does not contain any other special $O$-marking, and $n_i(H)$ is defined to be $1$ if $H$ contains $O_i$, and is defined to be $0$ otherwise. A hexagon $H$ joins a generator $x\in\mc{G}_{\mf{G}_1}$ to a generator $z\in\mc{G}_{\mf{G}_1}$, if $H$ does not contain any $x$-coordinate in its interior, and $\del(\del H|_{\al})=z-x$; the set of all hexagons joining $x$ to $z$ is denoted by $\mc{H}(x,z)$.  The chain map $f\colon\CF_{\mf{G}_1}\rightarrow \CF_{\mf{G}_2}$ is $U_i$-equivariant for all $i$, and for $x\in\mc{G}_{\mf{G}_1}$, it is defined as follows:
$$f(x)=
\begin{cases}
F_{11}^{-1}(x)& \text{if }F_{11}^{-1}(x)\neq\varnothing,\\
\sum_{y\in\mc{G}_{\mf{G}_2}}y\sum_{H\in\mc{H}(x,F_{11}(y))}\prod_i U_i^{n_i(H)}& \text{otherwise.}
\end{cases}
$$

It is not hard to see that this map is a chain map which increases the Maslov grading by $1$. In fact, this map is simply the map $F^L$ of \cite[Section 3.2]{CMPOZSzDT}.  It follows from \cite[Proposition 3.8]{CMPOZSzDT} that this map is surjective at the level of homology.

\subsubsection*{$S^3$-grid move (6)}\label{subsub:s3saddle} $\mf{G}_2$
is obtained from $\mf{G}_1$ by the following process: we assume that
$\mf{G}_1$ has exactly $(k+1)$ normal $O$-markings and a $2\times 2$
square $B$ which contains a special $O$-marking on the top-left and
$O_{k+1}$ on the bottom-right; $\mf{G}_2$ is obtained from $\mf{G}_1$
by deleting the two $O$-marking in $B$ and adding two special
$O$-markings, one on the bottom-left and one on the top-right of
$B$. We define the $U_i$-equivariant chain map
$f\colon\CF_{\mf{G}_1}\rightarrow \CF_{\mf{G}_2}$ as follows: set
$U_{k+1}=0$; send every generator that does not have a coordinate at
the center of $B$ to zero; and send every generator with a coordinate
at the center of $B$ to itself.  It is easy to see that this is a
chain map which drops Maslov grading by $1$. Furthermore, if we
compose this map $f$ with the map corresponding to
\hyperref[subsub:s3specialdestab]{$S^3$-grid move (5)}, we get one of
the normal destabilization maps, as described in \cite[Section
3.2]{CMPOZSzDT}. (We have not discussed this specific destabilization
map while discussing \hyperref[subsub:s3normaldestab]{$S^3$-grid move
  (4)}; but in our notation, this would have been the destabilization
map $d_{21}$.) Since the composition is a quasi-isomorphism, the map
$f$ must be injective at the level of homology; and indeed, this gives
an alternate proof of the fact that the map for
\hyperref[subsub:s3specialdestab]{$S^3$-grid move (5)} is surjective
at the level of homology.

\subsection{Moves on link-grid diagrams}
Quite like in the previous subsection, in this subsection we will analyse certain \emph{link-grid moves} which convert a link-grid diagram $\mf{L}_1$ to another link-grid diagram $\mf{L}_2$. In all the link-grid moves that we will analyse, the two $S^3$-grid diagrams $\mf{f}(\mf{L}_1)$ and $\mf{f}(\mf{L}_2)$ will already be related by one of the six $S^3$-grid moves; therefore, we already have maps $\CF_{\mf{f}(\mf{L}_1)}\rightarrow \CF_{\mf{f}(\mf{L}_2)}$. We will simply determine the \emph{Alexander grading shifts} in each case. The Alexander grading shift of a map $f\colon\CF_{\mf{f}(\mf{L}_1)}\rightarrow \CF_{\mf{f}(\mf{L}_2)}$ is the smallest possible $s\in\frac{1}{2}\Z$, such that $f$ shifts the Alexander grading of each element by at most $s$. In other words, for each $a\in\frac{1}{2}\Z$, there is the following commuting square.
$$\begin{tikzpicture}
\matrix(m)[matrix of math nodes,
row sep=3.5em, column sep=3em,
text height=2ex, text depth=0.25ex]
{\mc{F}_{\mf{L}_1}(a)&\mc{F}_{\mf{L}_2}(a+s)\\
\CF_{\mf{f}(\mf{L}_1)}& \CF_{\mf{f}(\mf{L}_2)}\\};
\path[right hook->](m-1-1) edge (m-2-1);
\path[right hook->](m-1-2) edge (m-2-2);
\path[->](m-1-1) edge (m-1-2);
\path[->](m-2-1) 
edge node[above]{$f$} (m-2-2);
\end{tikzpicture}$$

\subsubsection*{Link-grid move (1)}\label{subsub:linkrenumber} $\mf{L}_2$ is obtained from $\mf{L}_1$ by \emph{renumbering} the normal $O$-markings. This corresponds to the \hyperref[subsub:s3renumbering]{$S^3$-grid move $(2)$}, and the Alexander grading shift is $0$.

\subsubsection*{Link-grid move (2)}\label{subsub:linkcommute}
$\mf{L}_2$ is obtained from $\mf{L}_1$ by a
\emph{commutation}. Commutation comes in two flavors, horizontal
commutation and vertical commutation. In a horizontal commutation, we
choose two adjacent horizontal annuli such that the zero-sphere
obtained by projecting the markings in one annulus to the middle
$\al$-circle is unlinked from the zero-sphere obtained by projecting
the markings in the other annulus to that $\al$-circle; and then we
interchange the two horizontal annuli, cf.\
\hyperref[fig:linkcommutation]{Figure \ref*{fig:linkcommutation}}.  
In a vertical commutation, we
choose two adjacent vertical annuli such that the zero-sphere obtained
by projecting the markings in one annulus to the middle $\be$-circle
is unlinked from the zero-sphere obtained by projecting the markings
in the other annulus to that $\be$-circle, and then interchange the
two vertical annuli, cf.\ \cite[Figure 5]{CMPOZSzDT}. Commutation
corresponds to the \hyperref[subsub:s3commutation]{$S^3$-grid move
  $(3)$}, and as shown in \cite[Lemma 3.1]{CMPOZSzDT}, the Alexander
grading shift is $0$.

\begin{figure}
\psfrag{x}{$\XX$}
\psfrag{o}{$\OO$}
\psfrag{a}{$\al_2$}
\psfrag{a'}{$\al_1$}
\begin{center}
\includegraphics[width=0.5\textwidth]{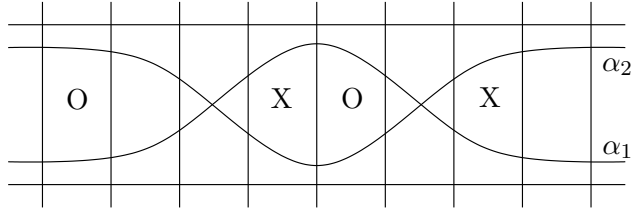}
\end{center}
\caption{A diagram representing a horizontal commutation. The
  link-grid diagram $\mf{L}_i$ is obtained from this diagram by
  deleting the the circle labeled $\al_i$.}\label{fig:linkcommutation}
\end{figure}

\subsubsection*{Link-grid move (3)}\label{subsub:linkdestab}
$\mf{L}_2$ is obtained from $\mf{L}_1$ by a \emph{destabilisation} or
vice-versa. In a destabilisation from $\mf{L}_1$ to $\mf{L}_2$, we
assume that $\mf{L}_1$ has exactly $(k+1)$ normal $O$-markings, and we
assume that there is a $2\times 2$ square $B$, three of whose squares
are occupied by $O_{k+1}$ and two $X$-markings. We then delete
$O_{k+1}$ and these two $X$-markings, and we put a new $X$-marking in
the other square of $B$. We then deformation retract the horizontal
annulus which contained $O_{k+1}$ to an $\al$-circle and deformation
retract the vertical annulus which contained $O_{k+1}$ to a
$\be$-circle to get the link-grid diagram $\mf{L}_2$. This move
corresponds to the \hyperref[subsub:s3normaldestab]{$S^3$-grid move
  $(4)$}, where we use the snail maps which are centered at the center
of $B$. As shown in \cite[Lemma 3.5]{CMPOZSzDT}, the Alexander grading
shift is $0$.

\subsubsection*{Link-grid move (4)}\label{subsub:linkbirth} $\mf{L}_2$
is obtained from $\mf{L}_1$ by a \emph{birth}, i.e.\ we assume that
$\mf{L}_2$ has exactly $(k+1)$ normal $O$-markings, with $O_{k+1}$
lying in the same square as some $X$-marking, and $\mf{L}_1$ is
obtained from $\mf{L}_2$ by deleting $O_{k+1}$ and that $X$-marking,
and then deformation retracting the horizontal and the vertical
annulus through that square to an $\al$-circle and a $\be$-circle,
respectively. This move corresponds to the
\hyperref[subsub:s3normaldestab]{$S^3$-grid move $(4)$} and it
represents a birth happening in a cobordism.

We will now show that the Alexander grading shift is $(-\frac{1}{2})$. Let us reuse the
notations from \hyperref[subsub:s3normaldestab]{$S^3$-grid move
  $(4)$}. There are two stabilisation maps $s_{11}$ and $s_{22}$ from
$\CF_{\mf{f}(\mf{L}_1)}$ to $\CF_{\mf{f}(\mf{L}_2)}$. We will only
deal with the map $s_{11}$; the map $s_{22}$ can be dealt with in a
similar fashion. Consider the map
$\wt{s}\colon\mc{G}_{\mf{f}(\mf{L}_2)}\rightarrow
\CF_{\mf{f}(\mf{L}_2)}$ defined as follows:
$$\wt{s}(x)=\sum_{y\in\mc{G}_{\mf{f}(\mf{L}_2)}}y\sum_{D\in\mc{S}_{33}(x,y,\al_1\cap\be_1)}\prod_{1\leq
  i\leq k} U_i^{n_i(D)}.$$
For any generator $x\in\mc{G}_{\mf{f}(\mf{L}_1)}$,
$s_{11}(x)=\wt{s}(F_{11}(x))$. The Alexander grading shift of the map
$\wt{s}$ is zero, since the Alexander grading shift induced by a
snail-like domain $D$ is the number of $O$'s minus the number of $X$'s
(both counted with multiplicities) that are contained in $D$; and
$O_{k+1}$ appears in $D$ the same number of times as the $X$-marking
that lies in the same square as $O_{k+1}$, and every other normal
$O$-marking $O_i$ appears with a cancelling factor of $U_i$,
cf.\ \cite[Proof of Lemma 3.5]{CMPOZSzDT}. 
Therefore, we only have to compute the Alexander grading shift of the
map $F_{11}\colon\mc{G}_{\mf{f}(\mf{L}_1)}\rightarrow \mc{G}_{\mf{f}(\mf{L}_2)}$.

Assume that the index of $\mf{L}_2$ is $(n+1)$. Let us cut up the
torus along $\al_1$, i.e.\ the $\al$-circle that lies just below
$O_{k+1}$, and $\be_2$, i.e.\ the $\be$-circle that lies just to the
right of $O_{k+1}$, to identify it with $[0,n+1)\times [0,n+1)$. The
$\al$-circles of $\mf{L}_2$ become the horizontal lines
$[0,n+1)\times\{i\}$ for $i\in\{0,\ldots,n\}$ and the $\be$-circles of
$\mf{L}_2$ become the vertical lines $\{i\}\times[0,n+1)$ for
$i\in\{0,\ldots,n\}$. To get from $\mf{L}_1$ to $\mf{L}_2$, we start
with the subsquare $[0,n)\times[1,n+1)$, we add the lowermost row
$[0,n+1)\times [0,1)$ and the rightmost column $[n,n+1)\times
[0,n+1)$, and we add an $X$-marking and the $O$-marking $O_{k+1}$ at the
bottom-right square $[n,n+1)\times[0,1)$. To get from
$x\in\mc{G}_{\mf{f}(\mf{L}_1)}$ to
$F_{11}(x)\in\mc{G}_{\mf{f}(\mf{L}_2)}$, we start with a formal sum of
$n$ points in $[0,n)\times [1,n+1)$ and we add the point $(n,0)$.

Recall that the Alexander grading of a generator $y$ is
$A(y)=\mc{J}(y,X)-\mc{J}(y,O)-\frac{1}{2}\mc{J}(X,X)+\frac{1}{2}\mc{J}(O,O)-\frac{n-1}{2}$. This
process increases each of the terms $\mc{J}(y,X)$, $\mc{J}(y,O)$ and
$\frac{n-1}{2}$ by $\frac{1}{2}$ and does not change the terms
$\mc{J}(X,X)$ and $\mc{J}(O,O)$. Therefore, the net Alexander grading
shift is $(-\frac{1}{2})$.

\subsubsection*{Link-grid move (5)}\label{subsub:linksaddle}
$\mf{L}_2$ is obtained from $\mf{L}_1$ by a \emph{saddle}, i.e.\ we
assume that $\mf{L}_1$ has a $2\times 2$ square $B$ with two
$X$-markings, one at the top-left and one at the bottom-right, and
$\mf{L}_2$ is obtained from $\mf{L}_1$ by deleting these two
$X$-markings and putting two new ones, one at the top-right and one at
the bottom-left of $B$. This move corresponds to the
\hyperref[subsub:s3identity]{$S^3$-grid move $(1)$}, and a direct
computation reveals that the Alexander grading shift is
$\frac{1}{2}$. This move represents a \hyperref[saddle]{saddle}
happening in a cobordism, as illustrated in
\hyperref[fig:saddle]{Figure \ref*{fig:saddle}}. 

Conversely, given a saddle from a link $L_1$ to a link $L_2$, we can
choose link-grid diagrams $\mf{L}_i$ representing $L_i$ such that
$\mf{L}_2$ is obtained from $\mf{L}_1$ by a saddle move. Represent the
two strands of $L_1$ where the saddle takes place by the configuration
as shown in the first part of \hyperref[fig:saddle]{Figure
  \ref*{fig:saddle}}; extend this to a rectilinear approximation for
the rest of $L_1$; in the resulting diagram, if there is a crossing
where the vertical arc is the overpass, perform the local adjusment
from \cite[Figure 7]{gridPC} to rectify it. This produces the
link-grid diagram $\mf{L}_1$ for $L_1$. Doing the saddle move to
$\mf{L}_1$ produces the link-grid diagram $\mf{L}_2$ for $L_2$.

\begin{figure}
\psfrag{x}{$\XX$}
\psfrag{o}{$\OO$}
\begin{center}
\includegraphics[width=0.7\textwidth]{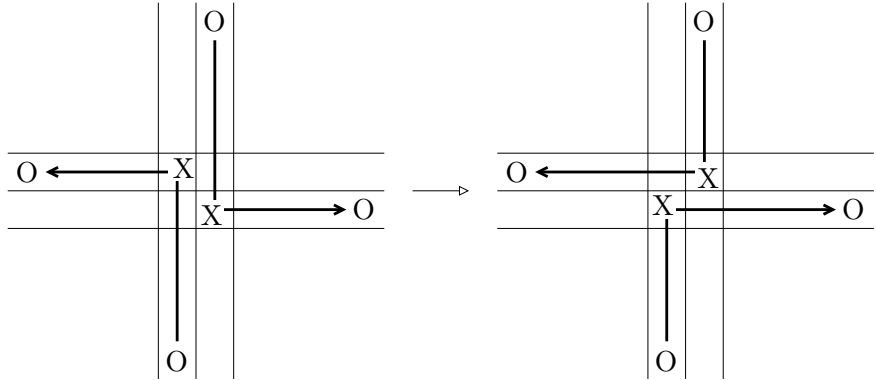}
\end{center}
\caption{The saddle move. The two link-grid diagrams $\mf{L}_1$ and
  $\mf{L}_2$ are shown, along with the oriented links that they
  represent (drawn with thick lines).}\label{fig:saddle}
\end{figure}

\subsubsection*{Link-grid move (6)}\label{subsub:linksplit} $\mf{L}_2$
is obtained from $\mf{L}_1$ by the following process: we assume that
$\mf{L}_1$ has exactly $(k+1)$ normal $O$-markings and a $2\times 2$
square $B$ with a special $O$-marking on the top-left and $O_{k+1}$ on
the bottom-right; $\mf{L}_2$ obtained from $\mf{L}_1$ by deleting
these two markings and adding two new special $O$-markings, one at the
top-right and one at the bottom-left of $B$. This corresponds to the
\hyperref[subsub:s3saddle]{$S^3$-grid move $(6)$}, and the Alexander
grading shift is $(-\frac{1}{2})$. This move also represents a saddle
in a cobordism (it will become apparent in the
\hyperref[mainproof]{proof of Theorem \ref*{thm:main}} why need two types of
saddle moves). Once again, it is easy to see that any saddle can be
represented by such a link-grid move.  Furthermore, if the saddle is a
\hyperref[split]{split}, then $\mf{L}_1$ is \hyperref[tight]{tight} if
and only if $\mf{L}_2$ is tight.

% The reader might wonder why we have two different link-grid moves
% for saddles. Given a connected cobordism $S$ from a link $L$
% represented by a link-grid diagram $\mf{L}$ to a link $L'$
% represented by a link-grid diagram $\mf{L}'$, we want to get a chain
% map from $\CF_{\mf{f}(\mf{L})}$ to $\CF_{\mf{f}(\mf{L}')}$ with an
% Alexander grading shift of $-\frac{\chi(S)}{2}$. The assumption that
% $S$ is connected is vital; otherwise, we might have sphere
% components. The fact that $S$ is connected will allow us to put $S$
% in a \hyperref[standard]{standard form}, where the link at a
% particular \hyperref[linkisknot]{still} is a knot. We will use
% link-grid move $(5)$ for all the saddles until that still, and we
% will use link-grid move $(6)$ for all the saddles after that still.

\subsubsection*{Link-grid move (7)}\label{subsub:linkdeath} $\mf{L}_2$ is obtained from $\mf{L}_1$ by a \emph{death}, i.e.\  there is some special $O$-marking in $\mf{L}_1$ such that the square immediately to the top-right of it contains an $X$-marking and a special $O$-marking, and $\mf{L}_2$ is obtained from $\mf{L}_1$ by deleting those two markings, and then deformation retracting the horizontal and vertical annulus through that square to an $\al$-circle and a $\be$-circle, respectively. This move corresponds to the \hyperref[subsub:s3specialdestab]{$S^3$-grid move $(5)$}. By direct computation, we see that the Alexander grading shift is $\frac{1}{2}$. This move represents a \hyperref[death]{death} happening in a cobordism.

\section{Main Theorem}

In this section, we will prove our main theorem, \hyperref[thm:main]{Theorem \ref*{thm:main}}.

\begin{lem}\label{lem:cromwell}
  If $\mf{L}_1$ and $\mf{L}_2$ are two link-grid diagrams representing   isotopic links, and if every link component in $\mf{L}_1$ contains   at most one special $O$-marking, and if the corresponding link   component in $\mf{L}_2$ contains the same number of special   $O$-markings, then there is a sequence of link-grid moves of Types   \hyperref[subsub:linkrenumber]{$(1)$},   \hyperref[subsub:linkcommute]{$(2)$} and   \hyperref[subsub:linkdestab]{$(3)$}, which take $\mf{L}_1$ to   $\mf{L}_2$.
\end{lem}

\begin{proof}
  This is a small extension of Cromwell's Theorem \cite[Section
  2]{gridPC}. Cromwell's theorem states that the above is true if all
  the $O$-markings are treated as equal. Therefore, we can simply take
  the sequence of link-grid moves as given by Cromwell, and apply
  them. However, we can run into the following four types of problems.
\begin{enumerate}
\item  We might have to destabilise at a special $O$-marking.
\item We might have to destabilise at a normal $O$-marking which is
  not the highest numbered one.
\item In the final link-grid diagram and in
  $\mf{L}_2$, the special $O$-markings could be at different places.
\item The normal $O$-markings could be numbered differently in
  the final link-grid diagram and in $\mf{L}_2$.
\end{enumerate}

\begin{figure}
\psfrag{x}[][][0.75]{$\XX$}
\psfrag{o}[][][0.75]{$\OO_k$}
\psfrag{oi}[][][0.75]{$\OO_i$}
\psfrag{os}{$\varnothing$}
\psfrag{1}{$(1)$}
\psfrag{21}{$(2)^{m_1}$}
\psfrag{22}{$(2)^{m_2}$}
\psfrag{23}{$(2)^{m_3}$}
\psfrag{24}{$(2)^{m_4}$}
\psfrag{3}{$(3)$}
\begin{center}
\includegraphics[width=\textwidth]{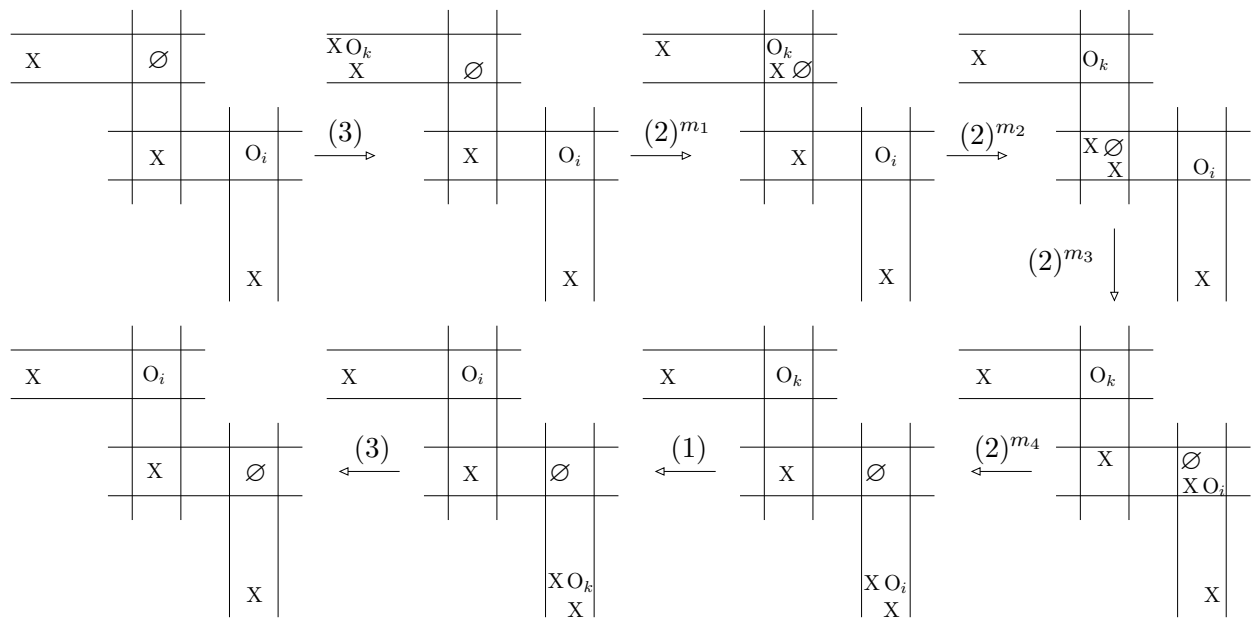}
\end{center}
\caption{A sequence of link-grid moves.}\label{fig:pointedsequence}
\end{figure}

The \hyperref[subsub:linkrenumber]{link-grid move $(1)$},
i.e.\ renumbering the normal $O$-markings, fixes two of these problems,
namely the second and the fourth one. To fix the other problems, we
only need a sequence of link-grid moves of Types
\hyperref[subsub:linkrenumber]{$(1)$},
\hyperref[subsub:linkcommute]{$(2)$} and
\hyperref[subsub:linkdestab]{$(3)$}, which achieves the following:
given a link-grid diagram $\mf{L}$ where every link component has at
most one special $O$-marking, and given a special $O$-marking, we can
convert that special $O$-marking to a normal $O$-marking, and convert
the next $O$-marking in that oriented link component to a special
$O$-marking. Assuming that there are exactly $(k-1)$ normal
$O$-markings in $\mf{L}$, such a sequence of moves is shown in
\hyperref[fig:pointedsequence]{Figure \ref*{fig:pointedsequence}}.
\end{proof}

\begin{thm}\label{thm:gridmoves}
  If $\mf{L}_1$ and $\mf{L}_2$ are two \hyperref[tight]{tight}   link-grid diagrams representing knots $K_1$ and $K_2$, respectively,   and if there is a connected knot cobordism from $K_1$ to $K_2$ with   exactly $b$ births, $s$ saddles and $d$ deaths, then there is a   sequence of link-grid moves taking $\mf{L}_1$ to $\mf{L}_2$, such   that there are exactly $b$ link-grid moves of Type \hyperref[subsub:linkbirth]{$(4)$}, $s-d$   link-grid moves of Type \hyperref[subsub:linksaddle]{$(5)$}, $d$ link-grid moves of Type \hyperref[subsub:linksplit]{$(6)$} and   $d$ link-grid moves of Type \hyperref[subsub:linkdeath]{$(7)$}, and these link-grid moves happen   in this order.
\end{thm}

\begin{proof}
  Using \hyperref[lem:isotopy]{Lemma \ref*{lem:isotopy}}, we can assume that all the births   happen at time $t=\frac{1}{4}$, all the saddles happen at time   $t=\frac{1}{2}$, and all the deaths happen at time   $t=\frac{3}{4}$. By isotoping the cobordism slightly, we can ensure   that the critical points happen at distinct instants of time, and we   can choose the order of the $b$ births, the order of the $s$ saddles   and the order of the $d$ deaths. Stated differently, given   $t_1<\cdots <t_b$ near $\frac{1}{4}$, $t_{b+1}<\cdots <t_{b+s}$ near   $\frac{1}{2}$, and $t_{b+s+1}<\cdots <t_{b+s+d}$ near $\frac{3}{4}$,   and given an ordering of the $b$ births, an ordering of the $s$   saddles and an ordering of the $d$ deaths, we can isotope the   cobordism slighly to ensure that that the $i$-th birth happens at   time $t=t_{i}$, the $i$-th saddle happens at time $t=t_{b+i}$ and   the $i$-th death happens at time $t=t_{b+s+i}$. Since the cobordism   is connected, we order the $s$ saddles in some way so as to   guarantee that the final $d$ saddles are all \hyperref[split]{splits}. In other words,   we ensure that the link in the still\phantomsection\label{linkisknot}, just after time $t=t_{b+s-d}$,   is a knot. A schematic picture of a cobordism, put in this standard   form\phantomsection\label{standard}, is shown in \hyperref[fig:example]{Figure \ref*{fig:example}}.

\begin{figure}
\psfrag{0}{$0$}
\psfrag{1}{$t_1$}
\psfrag{2}{$t_2$}
\psfrag{3}{$t_3$}
\psfrag{4}{$t_4$}
\psfrag{5}{$t_5$}
\psfrag{6}{$t_6$}
\psfrag{7}{$t_7$}
\psfrag{8}{$t_8$}
\psfrag{9}{$t_9$}
\psfrag{10}{$t_{10}$}
\psfrag{11}{$1$}
\begin{center}
\includegraphics[width=0.9\textwidth]{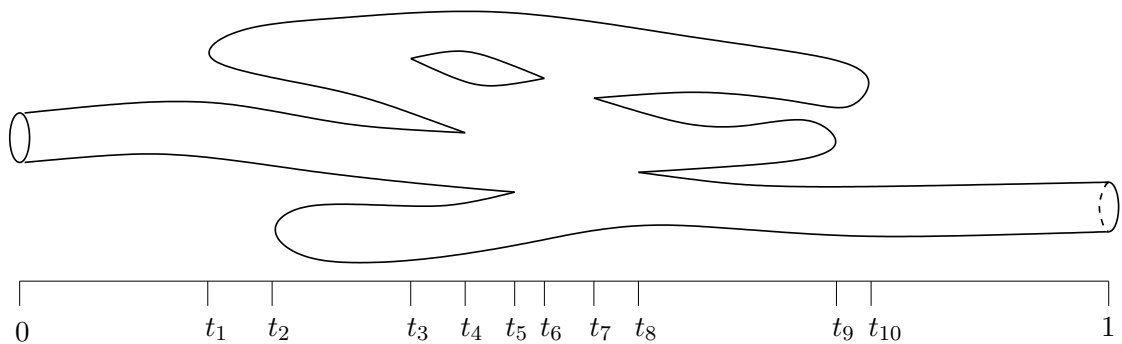}
\end{center}
\caption{A cobordism with $2$ births, $6$ saddles and     $2$ deaths, put in the standard form.}\label{fig:example}
\end{figure}

Let $\mf{L}^a_0=\mf{L}_1$ and let $\mf{L}^b_{b+s+d+1}=\mf{L}_2$.  For
each $1\leq i\leq b+s+d$, choose two link-grid diagrams $\mf{L}^b_i$
and $\mf{L}^a_i$, such that: $\mf{L}^b_i$ represents the link just
before time $t_i$; $\mf{L}^a_i$ represents the link just after time
$t_i$; and $\mf{L}^a_i$ can be obtained from $\mf{L}^b_i$ by a
link-grid move of Type $(4)$, $(5)$, $(6)$ or $(7)$, depending on
whether $1\leq i\leq b$, $b+1\leq i\leq b+s-d$, $b+s-d+1\leq i\leq
b+s$ or $b+s+1\leq i\leq b+s+d$, respectively. Observe that for each
$0\leq i\leq b+s+d$, the two link-grid diagrams $\mf{L}^a_i$ and
$\mf{L}^b_{i+1}$ represent isotopic links; it is easy to see that
while choosing $\mf{L}^a_i$ and $\mf{L}^b_{i+1}$, we can ensure that the
corresponding link components contain the same number of special
$O$-markings.  Therefore, by \hyperref[lem:cromwell]{Lemma
  \ref*{lem:cromwell}}, we can convert $\mf{L}^a_i$ to
$\mf{L}^b_{i+1}$ by a sequence of link-grid moves of Types
$(1)$-$(3)$. Putting everything together, we get the required sequence
of link-grid moves that converts $\mf{L}_1$ to $\mf{L}_2$.
\end{proof}

\begin{proof}[Proof of Theorem \ref*{thm:main}]\phantomsection\label{mainproof}
  In order to prove this, we only need to show the following: if
  $\mf{L}_1$ and $\mf{L}_2$ are two tight link-grid diagrams
  representing knots $K_1$ and $K_2$, respectively, and if there is a
  connected knot cobordism of genus $g$ from $K_1$ to $K_2$, then
  $\left|\tau_{\mf{L}_1}-\tau_{\mf{L}_2}\right|\leq g$.

  Let us now assume that the cobordism from $K_1$ to $K_2$ has $b$
  births, $d$ deaths and $2g+b+d$
  saddles. \hyperref[thm:gridmoves]{Theorem \ref*{thm:gridmoves}}
  tells us that there is a sequence of link-grid moves of Types
  $(1)$-$(7)$ taking $\mf{L}_1$ to $\mf{L}_2$, such that there are
  exactly $b$ link-grid moves of Type $(4)$, $2g+b$ link-grid moves of
  Type $(5)$, $d$ link-grid moves of Type $(6)$ and $d$ link-grid
  moves of Type $(7)$, and these link-grid moves happen in this order.

  For link-grid moves of Types $(1)$-$(5)$, the associated maps on   $\CF$ preserve Maslov gradings and are quasi-isomorphisms. For \hyperref[subsub:linksplit]{link-grid move $(6)$}, the map drops Maslov grading by $1$,   and is injective at the level of homology. The \hyperref[smallestgrading]{smallest Maslov grading} also drops by $1$, and by  \hyperref[thm:basicproperty]{Theorem \ref*{thm:basicproperty}}, the homology of $\CF$ in the smallest  Maslov grading is $\F_2$. Therefore, the map on   homology in the smallest Maslov grading is an isomorphism.   Similarly, for \hyperref[subsub:linkdeath]{link-grid move $(7)$}, the map increases   Maslov grading by $1$, and is surjective at the level of homology.   However, the smallest Maslov grading also increases by $1$;   therefore, the map on homology in the smallest Maslov grading is   also an isomorphism. Thus, the composed maps from   $\CF_{\mf{f}(\mf{L}_1)}$ to $\CF_{\mf{f}(\mf{L}_2)}$ preserves Maslov   grading and is a quasi-isomorphism in the smallest Maslov grading. However, each of the link-grid diagrams $\mf{L}_1$ and $\mf{L}_2$ contains exactly one special $O$-marking, so the smallest Maslov grading is the only Maslov grading in which the homology is supported. Therefore, the composed map is a quasi-isomorphism.

There are no Alexander grading shifts for link-grid moves \hyperref[subsub:linkrenumber]{$(1)$}, \hyperref[subsub:linkcommute]{$(2)$}
and \hyperref[subsub:linkdestab]{$(3)$}, there is an Alexander grading shift of $\frac{1}{2}$ for
link-grid moves \hyperref[subsub:linksaddle]{$(5)$} and \hyperref[subsub:linkdeath]{$(7)$}, and there is an Alexander grading
shift of $(-\frac{1}{2})$ for link-grid moves \hyperref[subsub:linkbirth]{$(4)$} and
\hyperref[subsub:linksplit]{$(6)$}. Therefore, the net Alexander grading shift is
$-\frac{b}{2}+\frac{2g+b}{2}-\frac{d}{2}+\frac{d}{2}=g$. Therefore,
for every $a\in\frac{1}{2}\Z$, we have the following commuting square.
$$\begin{tikzpicture}
\matrix(m)[matrix of math nodes,
row sep=3.5em, column sep=3em,
text height=2ex, text depth=0.25ex]
{H_*(\mc{F}_{\mf{L}_1}(a))&H_*(\mc{F}_{\mf{L}_2}(a+g))\\
H_*(\CF_{\mf{f}(\mf{L}_1)})=\F_2& H_*(\CF_{\mf{f}(\mf{L}_2)})=\F_2\\};
\path[->](m-1-1) edge (m-2-1);
\path[->](m-1-2) edge (m-2-2);
\path[->](m-1-1) edge (m-1-2);
\path[->](m-2-1) 
edge node[above]{$\Id$} (m-2-2);
\end{tikzpicture}$$

Substituting $a=\tau_{\mf{L}_1}$ in the above commutative diagram, we see that the map $H_*(\mc{F}_{\mf{L}_2}(\tau_{\mf{L}_1}+g))\rightarrow H_*(\CF_{\mf{f}(\mf{L}_2)})$ is non-trivial; therefore, $\tau_{\mf{L}_2}\leq\tau_{\mf{L}_1}+g$.

However, we can view the cobordism in reverse, i.e.\ we can run the movie backwards, to get a connected genus $g$ cobordism from $K_2$ to $K_1$. That would show that $\tau_{\mf{L}_1}\leq\tau_{\mf{L}_2}+g$. Combining the two inequalities, we get our desired result $\left|\tau_{\mf{L}_1}-\tau_{\mf{L}_2}\right|\leq g$.
\end{proof}

\section{Applications}

Given any $S^3$-grid diagram $\mf{G}$, define the \emph{special
  generator} $x_{\mf{G}}\in\mc{G}_{\mf{G}}$ to be the generator whose
coordinates lie at the top-left corners of the squares that contain
the $O$-markings. It is an easy computation to show that the Maslov
grading of $x_{\mf{G}}$ is always zero. 

Let $\mf{G}_n$ be the following index-$n$ $S^3$-grid diagram: there is
exactly one special $O$-marking; the square containing $O_{n-1}$ lies
immediately to the top-left of the square containing the special
$O$-marking; and for all $1\leq i\leq n-2$, the square containing
$O_i$ lies immediately to the top-left of the square containing
$O_{i+1}$. For this grid diagram, the special generator $x_{\mf{G}_n}$
also has coordinates at the bottom-right corners of the squares that
contain the $O$-markings. 

\begin{lem}
The $\F_2$-module generated by $x_{\mf{G}_n}$ is a direct summand of
$\CF_{\mf{G}_n}$.
\end{lem}

\begin{proof}
For any $S^3$-grid diagram $\mf{G}$, the $\F_2$-module generated by
$x_{\mf{G}}$ is a quotient complex of $\CF_{\mf{G}}$. This is because
any rectangle that joins some other generator to $x_{\mf{G}}$ must
pass through some $O$-marking. However, rectangles are not allowed to
pass through the special $O$-markings, and whenever they pass through
the normal $O$-markings, they pick up a $U$-power.

For the grid diagram $\mf{G}_n$, we would like to show that the $\F_2$-module generated by $x_{\mf{G}_n}$ is also a subcomplex. Let $y\in\mc{G}_{\mf{G}_n}$ be some generator that differs from $x_{\mf{G}_n}$ in exactly two coordinates. There are exactly two embedded rectangles $R_1$ and $R_2$, with boundary lying in $\al\cup\be$, whose top-right and bottom-left corners are $x_{\mf{G}_n}$-coordinates, and whose top-left and bottom-right corners are $y$-coordinates. It is clear that none of these rectangles contain any $O$-markings or any $x_{\mf{G}_n}$-coordinates in their interiors.  Therefore, $\del x_{\mf{G}_n}=0$, thereby concluding the proof.
\end{proof}

\begin{lem}\label{lem:taualex}
Let $\mf{L}$ be an index-$n$ link-grid diagram that represents a knot
$K$, such that $\mf{f}(\mf{L})=\mf{G}_n$. Then $\tau(K)=A(x_{\mf{G}_n})$.
\end{lem}

\begin{proof}
  The link-grid diagram $\mf{L}$ is a tight link-grid diagram representing $K$, therefore, $\tau(K)=\tau_{\mf{L}}$, which is the smallest $a\in\frac{1}{2}\Z$, such that the map induced on homology from the inclusion $\mc{F}_{\mf{L}}(a)\hookrightarrow \CF_{\mf{G}_n}$ is non-trivial. However, the homology of $\CF_{\mf{G}_n}$ is one-dimensional, carried by the direct summand which is the $\F_2$-module generated by $x_{\mf{G}_n}$. Therefore $\tau_{\mf{L}}=A(x_{\mf{G}_n})$.
\end{proof}

\begin{figure}
\psfrag{x}{$\XX$}
\psfrag{o}{$\varnothing$}
\psfrag{o1}{$\OO_1$}
\psfrag{o2}{$\OO_2$}
\psfrag{o3}{$\OO_3$}
\psfrag{o4}{$\OO_4$}
\psfrag{o5}{$\OO_5$}
\psfrag{o6}{$\OO_6$}
\psfrag{o7}{$\OO_7$}
\begin{center}
\includegraphics[width=0.4\textwidth]{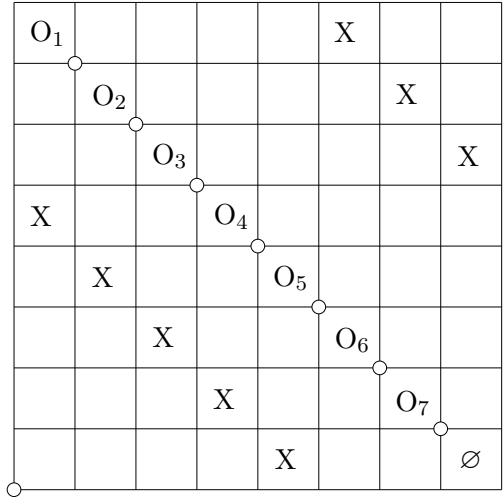}
\end{center}
\caption{The link-grid diagram $\mf{L}_{p,q}$ for $T(p,q)$, with $p=5$
  and $q=3$. The coordinates of the special generator
  $x_{\mf{G}_{p+q}}$ are shown.}\label{fig:torus}
\end{figure}

Let $T(p,q)$ denote the $(p,q)$-torus knot. We will represent $T(p,q)$
by the following index-$(p+q)$ link-grid diagram $\mf{L}_{p,q}$:
$\mf{f}(\mf{L}_{p,q})=\mf{G}_{p+q}$; squares to the bottom-right of
squares containing $X$-markings also contain $X$-markings; and the
$X$-marking in the horizontal annulus through the special $O$-marking,
lies $p$ squares to the right of the special $O$-marking. The
link-grid diagram $\mf{L}_{5,3}$ is shown in
\hyperref[fig:torus]{Figure \ref*{fig:torus}}. To draw the torus knot
$T(p,q)$ that $\mf{L}_{p,q}$ represents or to compute Alexander
gradings of specific generators, we need to identify $\mf{L}_{p,q}$
with a diagram on $[0,p+q)\times[0,p+q)$. For such identifications, we
always assume that the special $O$-marking lies in the bottom-right
square.

\begin{thm}
  There is an unknotting sequence for $T(p,q)$ with
  $\frac{(p-1)(q-1)}{2}$ crossing changes, and
  $\tau(T(p,q))=\frac{(p-1)(q-1)}{2}$. Therefore,
  $u(T(p,q))=g^*(T(p,q))=\frac{(p-1)(q-1)}{2}$.
\end{thm}

\begin{proof}
  Without loss of generality, let us assume that $p>q$. After
  identifying $\mf{L}_{p,q}$ with a picture on $[0,p+q)\times[0,p+q)$,
  let us consider the induced knot diagram for $T(p,q)$, cf.\ the
  first part of \hyperref[fig:unknotting]{Figure
    \ref*{fig:unknotting}}. In this picture, there are
  $\frac{q(q-1)}{2}$ crossings above the principal diagonal, and
  $(p-q)(q-1)+\frac{q(q-1)}{2}$ crossings below the principal
  diagonal, totalling to $p(q-1)$ crossings (thus, this is actually a
  minimal crossing diagram for $T(p,q)$).

\begin{figure}
\begin{center}
\includegraphics[width=\textwidth]{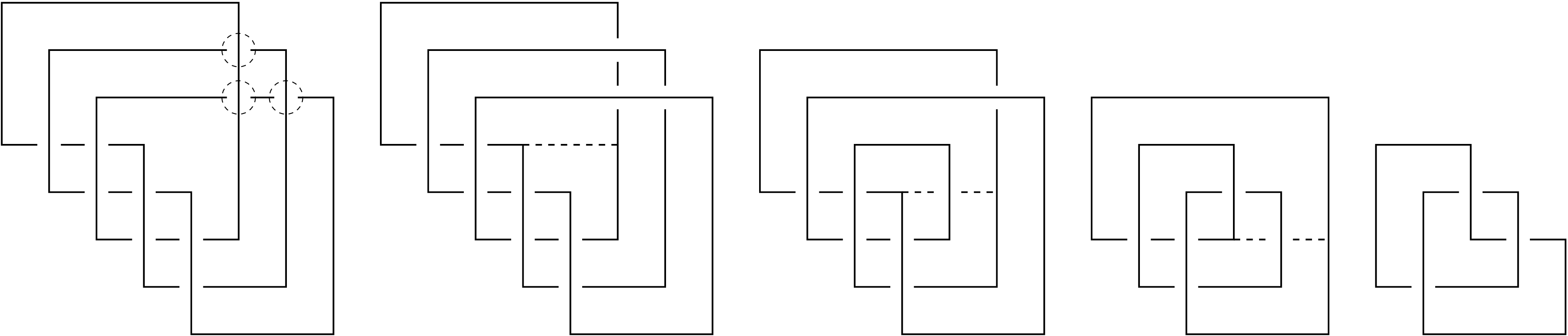}
\end{center}
\caption{Starting at the knot diagram induced by $\mf{L}_{p,q}$ with $p>q$, we do
  $\frac{q(q-1)}{2}$ crossing changes to get the knot diagram induced
  by $\mf{L}_{p-q,q}$.}\label{fig:unknotting}
\end{figure}

Change the $\frac{q(q-1)}{2}$ crossings above the principal
diagonal. \hyperref[fig:unknotting]{Figure
    \ref*{fig:unknotting}} shows how we can isotope the resulting knot
  diagram to get the diagram that would be induced by
  $\mf{L}_{p-q,q}$. However, by induction, $T(p-q,q)$ can be unknotted
  with $\frac{(p-q-1)(q-1)}{2}$ crossing changes. Therefore, $T(p,q)$
  can be unknotted with
  $\frac{(p-q-1)(q-1)}{2}+\frac{q(q-1)}{2}=\frac{(p-1)(q-1)}{2}$
  crossing changes.

  To compute $\tau(T(p,q))$, thanks to \hyperref[lem:taualex]{Lemma
    \ref*{lem:taualex}}, we only need to compute the Alexander grading
  of the special generator $x_{\mf{G}_{p+q}}$. Towards this end, let
  us number the coordinates of $x_{\mf{G}_{p+q}}$ from left to right
 as $x_1,\ldots,x_{p+q}$; let us number the $X$-markings from left to
 right as $X_1,\ldots,X_{p+q}$; and let us number the special
 $O$-marking as $O_{p+q}$. Then,
\begin{align*}
  A(x_{\mf{G}_{p+q}})&=\mc{J}(x_{\mf{G}_{p+q}},X)-\mc{J}(x_{\mf{G}_{p+q}},O)-\frac{1}{2}\mc{J}(X,X)+\frac{1}{2}\mc{J}(O,O)-\frac{p+q-1}{2}\\
  &=\sum_{1\leq i,j\leq p+q}\mc{J}(x_i,X_j)-\sum_{1\leq i,j\leq
    p+q}\mc{J}(x_i,O_j)-\sum_{1\leq i<j\leq p+q}\mc{J}(X_i,X_j)\\
  &\qquad {}+\sum_{1\leq i<j\leq p+q}\mc{J}(O_i,O_j)-\frac{p+q-1}{2}\\
  &=\sum_{j=1}^{p+q}\mc{J}(x_1,X_j)+\sum_{j=1}^p\sum_{i=2}^{p+q}\mc{J}(x_i,X_j)+\sum_{j=p+1}^{p+q}
  \sum_{i=2}^{p+q}\mc{J}(x_i,X_j)\\
  &\qquad {}-\sum_{j=1}^{p+q}\mc{J}(x_1,O_j)-\sum_{\substack{2\leq
      i\leq p+q\\1\leq j\leq
      p+q}}\mc{J}(x_i,O_j)-\sum_{\substack{1\leq i\leq p\\p+1\leq
      j\leq p+q}}\mc{J}(X_i,X_j)\\
  &\qquad {}-\sum_{\substack{1\leq i<j\leq
      p+q\\i>p\text{ or }j\leq p}}\mc{J}(X_i,X_j)+\sum_{1\leq i<j\leq
    p+q}\mc{J}(O_i,O_j)-\frac{p+q-1}{2}\\
  &=\frac{p+q}{2}+\frac{pq}{2}+\frac{pq}{2}-\frac{p+q}{2}-0-\frac{pq}{2}-0+0-\frac{p+q-1}{2}\\
  &=\frac{(p-1)(q-1)}{2}.
\end{align*}

Therefore,
$\tau(T(p,q))=A(x_{\mf{G}_{p+q}})=\frac{(p-1)(q-1)}{2}$. Combining our
results, we get $\frac{(p-1)(q-1)}{2}\geq u(T(p,q))\geq
g^*(T(p,q))\geq \tau(T(p,q))=\frac{(p-1)(q-1)}{2}$, thus completing
the proof.
\end{proof}

\section*{Acknowledgement}
The work was done when the author was supported by the Clay Research
Fellowship. He would like to thank John Baldwin, Robert Lipshitz,
Peter \Ozsvath{} and \Zoltan{} \Szabo{} for several helpful
discussions. He would also like to thank the referee for many helpful comments.

\bibliographystyle{amsalpha}
\bibliography{tau}

\providecommand{\bysame}{\leavevmode\hbox to3em{\hrulefill}\thinspace}
\providecommand{\MR}{\relax\ifhmode\unskip\space\fi MR }
\providecommand{\MRhref}[2]{%
  \href{http://www.ams.org/mathscinet-getitem?mr=#1}{#2}
}
\providecommand{\href}[2]{#2}
\begin{thebibliography}{MOSzT07}

\bibitem[BM94]{gridJBWM}
Joan Birman and William Menasco, \emph{Special positions for essential tori in
  link complements}, Topology \textbf{33} (1994), no.~3, 525--556.

\bibitem[Bru98]{gridHB}
H.~Brunn, \emph{\"{U}ber verknotete {K}urven}, Verhandlungen des
  Internationalen Math Kongresses (Zurich 1897) (1898), 256--259.

\bibitem[Cro95]{gridPC}
Peter Cromwell, \emph{Embedding knots and links in an open book {I}: {B}asic
  properties}, Topology and its Applications \textbf{64} (1995), no.~1, 37--58.

\bibitem[Dyn06]{gridID}
Ivan Dynnikov, \emph{Arc-presentation of links: {M}onotonic simplification},
  Fundamenta Mathematicae \textbf{190} (2006), 29--76.

\bibitem[KM93]{generalPKTM}
Peter Kronheimer and Tomasz Mrowka, \emph{Gauge theory for embedded surfaces
  {I}}, Topology \textbf{32} (1993), no.~4, 773--826.

\bibitem[Lyo80]{gridHCL}
Herbert~C. Lyon, \emph{Torus knots in the complement of links and surfaces},
  Michigan Math Journal \textbf{27} (1980), no.~1, 39--46.

\bibitem[Mat06]{gridHM}
Hiroshi Matsuda, \emph{Links in an open book decomposition and in the standard
  contact structure}, Proceedings of the American Mathematical Society
  \textbf{134} (2006), 3697--3702.

\bibitem[MOS09]{CMPOSS}
Ciprian Manolescu, Peter Ozsv\'{a}th, and Sucharit Sarkar, \emph{A
  combinatorial description of knot {F}loer homology}, Annals of Mathematics
  \textbf{169} (2009), no.~2, 633--660.

\bibitem[MOSzT07]{CMPOZSzDT}
Ciprian Manolescu, Peter Ozsv\'{a}th, Zolt\'{a}n Szab\'{o}, and Dylan Thurston,
  \emph{On combinatorial link {F}loer homology}, Geometry and Topology
  \textbf{11} (2007), 2339--2412.

\bibitem[Neu84]{gridLN}
Lee Neuwirth, \emph{*-projections of knots}, vol. Global Differential Geometry,
  Teubner-Texte zur Mathematik, no.~70, 1984.

\bibitem[Ni07]{YN}
Yi~Ni, \emph{Knot {F}loer homology detects fibred knots}, Inventiones
  Mathematicae \textbf{170} (2007), no.~3, 577--608.

\bibitem[OSz03]{POZSz4ballgenus}
Peter Ozsv\'{a}th and Zolt\'{a}n Szab\'{o}, \emph{Knot {F}loer homology and the
  four-ball genus}, Geometry and Topology \textbf{7} (2003), 625--639.

\bibitem[OSz04a]{POZSzgenusbounds}
\bysame, \emph{Holomorphic disks and genus bounds}, Geometry and Topology
  \textbf{8} (2004), 311--334.

\bibitem[OSz04b]{POZSzknotinvariants}
\bysame, \emph{Holomorphic disks and knot invariants}, Advances in Mathematics
  \textbf{186} (2004), no.~1, 58--116.

\bibitem[OSz08]{POZSzlinkinvariants}
\bysame, \emph{Holomorphic disks, link invariants and the multi-variable
  {A}lexander polynomial}, Algebraic and Geometric Topology \textbf{8} (2008),
  615--692.

\bibitem[OSzT08]{POZSzDT}
Peter Ozsv\'{a}th, Zolt\'{a}n Szab\'{o}, and Dylan Thurston, \emph{Legendrian
  knots, transverse knots, and combinatorial link {F}loer homology}, Geometry
  and Topology \textbf{12} (2008), 941--980.

\bibitem[Ras03]{JR}
Jacob Rasmussen, \emph{Floer homology and knot complements}, Ph.D. thesis,
  Harvard University, 2003.

\bibitem[Ras10]{generalJR}
\bysame, \emph{Khovanov homology and the slice genus}, Inventiones Mathematicae
  \textbf{182} (2010), no.~2, 419--447.

\bibitem[Rud92]{gridLR}
Lee Rudolph, \emph{Quasipositive annuli (constructions of quasipositive knots
  and links {I}{V})}, Journal of Knot Theory and its Ramifications \textbf{4}
  (1992), 451--466.

\end{thebibliography}

\end{document}